\documentclass{amsart}
\usepackage[boxruled, noend]{algorithm2e}
\usepackage{amssymb}
\usepackage{tikz-cd}
\usepackage{mathbbol} 
\usepackage{csquotes}

\usepackage[backend=biber,
	style=alphabetic,
	backref=true,
	url=false,
	doi=false,
	isbn=false,
	eprint=false
	]{biblatex}

\renewbibmacro{in:}{}  
\AtEveryBibitem{\clearfield{pages}}  

\DeclareFieldFormat{title}{\myhref{\mkbibemph{#1}}}
\DeclareFieldFormat
[article,inbook,incollection,inproceedings,patent,thesis,unpublished]
{title}{\myhref{\mkbibquote{#1\isdot}}}

\newcommand{\doiorurl}{%
	\iffieldundef{url}
	{\iffieldundef{eprint}
		{}
		{http://arxiv.org/abs/\strfield{eprint}}}
	{\strfield{url}}%
}

\newcommand{\myhref}[1]{%
	\ifboolexpr{%
		test {\ifhyperref}
		and
		not test {\iftoggle{bbx:eprint}}
		and
		not test {\iftoggle{bbx:url}}
	}
	{\href{\doiorurl}{#1}}
	{#1}%
}

\usepackage[bookmarks=true,
	linktocpage=true,
	bookmarksnumbered=true,
	breaklinks=true,
	pdfstartview=FitH,
	hyperfigures=false,
	plainpages=false,
	naturalnames=true,
	colorlinks=true,
	pagebackref=false,
	pdfpagelabels]{hyperref}

\hypersetup{
	colorlinks,
	citecolor=blue,
	filecolor=blue,
	linkcolor=blue,
	urlcolor=blue
}

\usepackage[capitalize, noabbrev]{cleveref}
\let\subsectionSymbol\S
\crefname{subsection}{\subsectionSymbol\!}{subsections}

\calclayout

\makeatletter
\@namedef{subjclassname@2020}{%
	\textup{2020} Mathematics Subject Classification}
\makeatother

\newtheorem{theorem}{Theorem}

\newtheorem{lemma}[theorem]{Lemma}

\theoremstyle{definition}
\newtheorem{definition}[theorem]{Definition}
\newtheorem{example}[theorem]{Example}
\newtheorem{remark}[theorem]{Remark}

\newtheorem*{notation}{Notation}

\DeclareMathOperator{\bd}{\partial}

\newcommand{\ot}{\otimes}
\DeclareMathOperator{\EZ}{EZ}
\DeclareMathOperator{\AW}{AW}

\newcommand{\id}{\mathsf{id}}
\newcommand{\Med}{\cM}

\newcommand{\N}{\mathbb{N}}
\newcommand{\Z}{\mathbb{Z}}

\renewcommand{\k}{\Bbbk}
\newcommand{\Sym}{\mathbb{S}}

\newcommand{\Ftwo}{{\mathbb{F}_2}}
\newcommand{\Fp}{{\mathbb{F}_p}}

\renewcommand{\th}{\mathrm{th}}

\newcommand{\Hom}{\mathrm{Hom}}

\newcommand{\bars}[1]{\lvert#1\rvert}

\newcommand{\cM}{\mathcal{M}}

\renewcommand{\P}{\mathrm{P}}
\newcommand{\RP}{\mathbb{R}\mathrm{P}}
\newcommand{\ind}{\mathrm{ind}}
\newcommand{\chain}[1]{C_{#1}}
\newcommand{\suspension}{\Sigma}
\newcommand{\sm}{\smallsetminus}
\DeclareMathOperator{\xor}{\triangle}

\DeclareMathOperator{\pos}{pos}
\newcommand{\bu}{{\bar{u}}}
\DeclareMathOperator{\modulo}{mod}
\DeclareMathOperator{\img}{img}
\DeclareMathOperator{\supp}{supp}
\DeclareMathOperator{\domain}{dom}

\newenvironment{tab}{\list{}{\rightmargin 0pt}\item\relax}{\endlist}
\title[new formulas for cup-$i$ products]{New formulas for cup-$i$ products and fast computation of Steenrod squares}
\author{Anibal M. Medina-Mardones}
\address{Max Planck Institute for Mathematics \and University of Notre Dame}
\email{\href{mailto:ammedmar@mpim-bonn.mpg.de}{ammedmar@mpim-bonn.mpg.de}}

\begin{document}

\begin{abstract}
	Operations on the cohomology of spaces are important tools enhancing the descriptive power of this computable invariant.
	For cohomology with mod 2 coefficients, Steenrod squares are the most significant of these operations.
	Their effective computation relies on formulas defining a \mbox{cup-$i$} construction, a structure on (co)chains which is important in its own right, having connections to lattice field theory, convex geometry and higher category theory among others.
	In this article we present new formulas defining a \mbox{cup-$i$} construction, and use them to introduce a fast algorithm for the computation of Steenrod squares on the cohomology of finite simplicial complexes.
	In forthcoming work we use these formulas to axiomatically characterize the cup-$i$ construction they define, showing additionally that all other formulas in the literature define the same cup-$i$ construction up to isomorphism.
\end{abstract}
	\maketitle

\section{Introduction}

Discrete models are indispensable for effective computations involving topological spaces.
The category of simplicial complexes provides models not only for spaces but also, through the simplicial approximation theorem, for continuous maps between them.
We can obtain algebraic models from these via simplicial chains and their dual cochains, from which Betti numbers can be readily computed using linear algebra alone.

In this article we focus on finer invariants of spaces enriching their mod 2 cohomology and going beyond Betti numbers.
We are referring to the celebrated Steenrod squares
\begin{equation*}
Sq^k \colon H^\bullet(X; \Ftwo) \to H^\bullet(X; \Ftwo).
\end{equation*}
These operations can be thought of as arising from the broken $\Sym_2$-symmetry of the diagonal map
\begin{equation*}
\begin{tikzcd}[column sep=small, row sep=-3pt]
X \arrow[r] & X \times X \\
x \arrow[r, mapsto] & (x,x)
\end{tikzcd}
\end{equation*}
occurring during the passage from continuous descriptions to discrete/algebraic models.

We mention the following examples to illustrate the additional discriminatory power these operations provide:
\begin{enumerate}
	\item The real projective plane and the wedge of a circle and a sphere have, with $\Ftwo$-coefficients, the same Betti numbers, yet the rank of $Sq^1$ tells them apart.
	\item Similarly, the complex projective plane and the wedge of a 2-sphere and a 4-sphere have the same Betti numbers with any coefficients, yet the rank of $Sq^2$ distinguishes them.
	\item The suspensions of the two spaces above have the same Betti numbers and also isomorphic cohomology rings, yet the rank of $Sq^2$ tells them apart.
\end{enumerate}

For simplicial complexes, effective constructions of Steenrod squares have been known since their introduction in Steenrod's seminal 1947 paper \cite{steenrod1947products}.
They all rely on a \textit{cup-$i$ construction}, a structure on chains given by a collection of natural linear maps
\begin{equation*}
\Delta_i \colon C_\bullet(X; \Ftwo)  \to C_\bullet(X; \Ftwo)^{\ot 2}\,,
\end{equation*}
satisfying for every integer $i$ the following key identity:
\begin{equation*}
(1+T) \Delta_{i-1} =
\partial \circ \Delta_i + \Delta_i \circ \partial,
\end{equation*}
where $T$ denoted the transposition of tensor factors,
and such that $\Delta_0$ is a chain approximation to the diagonal of $X$.
These \textit{cup-$i$ coproducts} and their linear dual \textit{cup-$i$ products} are important in their own right.
For example, they are used to describe action functionals of topological field theories \cite{gaiotto2016spin,kapustin2017fermionic,barkeshli2021classification}, to define the nerve of $n$-categories \cite{medina2020globular}, and their comodules can be used to fully faithfully model chain complex valued presheaves \cite{medina2022assembly} on $X$.

In this article we introduce new formulas defining a cup-$i$ construction on simplicial complexes and simplicial sets, a categorical closure of simplicial complexes used, for example, to define the singular homology of topological spaces.

Several formulas defining cup-$i$ constructions have been given in the literature starting with Steenrod's original \cite{steenrod1947products}.
These include those resulting from the approach of Real \cite{real1996computability} and Gonz\'alez-D\'iaz--Real \cite{gonzalez-diaz1999steenrod, gonzalez2003computation, gonzalez-diaz2005cocyclic} based on the EZ-AW chain contraction, the operadic methods of McClure-Smith \cite{mcclure2003multivariable} and Berger-Fresse \cite{berger2004combinatorial}, and the prop viewpoint of the author \cite{medina2020prop1, medina2021prop2}.
The question of comparing the resulting cup-$i$ constructions will be addressed via an axiomatic characterization in \cite{medina2022axiomatic}, where it is shown that all of these cup-$i$ constructions, including the one given here, are isomorphic and not just homotopic.

We highlight three uses for the formulas introduced in this paper.
1) They are key to prove the axiomatic characterization of Steenrod's cup-$i$ construction.
2) In \cite{cantero-moran2020khovanov}, Cantero-Mor\'an defined Steenrod squares in mod 2 Khovanov homology \cite{khovanov2000categorification} by reinterpreting them in the context of augmented semi-simplicial objects in the Burnside category.
3) They lead to fast computations of Steenrod square as we describe next.

Given a cup-$i$ construction and a finite simplicial complex, a representative of $Sq^k \big( [\alpha] \big)$ for a cocycle $\alpha$ is given by the cocycle $\beta = (\alpha \ot \alpha) \triangle_{i}(-)$ where $i$ is an integer that depends only on the degree of $\alpha$ and $k$.
A direct algorithmic way to compute the support of $\beta$ is to iterate over all simplices $x$ of the appropriate dimension, compute $\triangle_i(x)$, and record $x$ if the value of $(\alpha \ot \alpha)$ on it is $1 \in \Ftwo$.
Our algorithm improves on this scheme by considering only simplices $x$ related to the support of $\alpha$.
More specifically, it constructs the universal support of $\beta$ and then discards simplices in it that are not in $X$.
In this way our algorithm depends primarily on the size of the support of $\alpha$, and is therefore less sensitive to the number of simplices of $X$.

For the effective computation of Steenrod squares on simplicial complexes, an algorithm based on \cite{gonzalez-diaz1999steenrod} was implemented in the open-source mathematics system \verb|SAGE| by John Palmieri \cite{sagemath}.
We present a proof-of-concept performance comparison between a \verb|Python| implementation of our algorithm and the one in \verb|SAGE|.
The speed gained with our algorithm is essential for the incorporation of Steenrod squares into persistence homology \cite{medina2022per_st}, a technique typically used in highly intensive data analysis tasks \cite{carlsson2008images, chan2013viral, lee2017quantifying} and for which various software projects exist \cite{bauer2021ripser, gudhi, medina2021giotto}.
A specific implementation for the computation of Steenrod barcodes based on the algorithms introduced here can be found in the project \texttt{steenroder}\footnote{Currently hosted at \url{https://github.com/Steenroder/steenroder}}.

\subsection*{Outline}

In \cref{s:preliminaries} we review the notions from equivariant homological algebra and simplicial topology needed to present, in \cref{s:squares}, the definitions of cup-$i$ constructions and Steenrod squares.
We introduce our new formulas in \cref{s:formulas} deferring the proof that they define a cup-$i$ construction to \cref{s:proof}.
We present our algorithm in \cref{s:algorithm} and a proof of its correctness in \cref{s:correctness}.
We devote \cref{s:comparison} to a proof-of-concept comparison of our method using \verb|SAGE|.
In \cref{s:outlook} we discuss finer invariants associated to Steenrod squares, and provide conclusions and an outline for future work in \cref{s:conclusion}.

\subsection*{Acknowledgments}

We would like to thank Mark Behrens, Greg Brumfiel, Tim Campion, Federico Cantero-Mor\'an, Roc\'io Gonz\'alez-D\'iaz, Kathryn Hess, Riley Levy, Umberto Lupo, John Morgan, Pedro Real, Stephan Stolz, Dennis Sullivan, and Guillaume Tauzin, for stimulating conversations about this project.
We thank the anonymous referees for several suggestions improving the exposition of this work.

We are grateful for the hospitality of the \textit{Laboratory for Topology and Neuroscience} at EPFL, where part of this work was carried out, and acknowledge partial financial support from \textit{Innosuisse} grant 32875.1 IP-ICT-1.

\section{Preliminaries} \label{s:preliminaries}

In this section we review the basic notions used in this article and set up the conventions we follow.

\subsection{Chain complexes}

We assume familiarity with the notion of chain complex over a ring $\k$.

The \textit{tensor product} $C \ot C^\prime$ of chain complexes $C$ and $C^\prime$ is the chain complex whose degree-$n$ part is
\begin{equation*}
\left(C \ot C^\prime\right)_n = \bigoplus_{i+j=n} C_i \ot C^\prime_j,
\end{equation*}
where $C_i \ot C^\prime_j$ is the tensor product of $\k$-modules, and whose boundary map is defined by
\begin{equation*}
\partial (v \ot w) = \partial v \ot w + (-1)^{|v|} v \ot \partial w.
\end{equation*}

The \textit{hom complex} $\Hom(C, C^\prime)$ is the chain complex whose degree-$n$ part is the subset of linear maps between them that shift degree by $n$, i.e.,
\begin{equation*}
\Hom(C, C^\prime)_n = \{f \mid \forall k \in \Z, f(C_k) \subseteq C^\prime_{k+n}\},
\end{equation*}
and boundary map defined by
\begin{equation*}
\partial f =
\partial_{C^\prime} \circ f - (-1)^{\bars{f}} f \circ \partial_C.
\end{equation*}
Notice that a chain map is the same as a $0$-cycle in this complex, and that two chain maps are chain homotopy equivalent if and only if they are homologous cycles.
We extend this terminology and say that two maps $f, g \in \Hom(C, C^\prime)$ are \textit{homotopic} if their difference is nullhomologous, referring to a map $h \in \Hom(C, C^\prime)$ such that $\partial h = f - g$ as a \textit{homotopy} between $f$ and $g$.

Regarding $\k$ as a chain complex concentrated in degree $0$, the \textit{linear dual} of a chain complex $C$ is the chain complex $\Hom(C, \k)$.
We refer to the contravariant functor $\Hom(-, \k)$ as \textit{linear duality}.

For any three chain complexes, there is a natural isomorphism of chain complexes
\begin{equation} \label{e:adjunction isomorphism}
\Hom(C \ot C^\prime, C^{\prime\prime}) \cong
\Hom(C, \Hom(C^\prime, C^{\prime\prime}))
\end{equation}
referred to as the \textit{adjunction isomorphism}.

\subsection{Group actions}

Symmetries on chain complexes play an important role on this work.
Let $G$ be a finite group.
We will later focus solely on the symmetric group $\Sym_2$.
We denote by $\k[G]$ the group ring of $G$, i.e., the free $\k$-module generated by $G$ together with the ring product defined by linearly extending the product on $G$.
We refer to a chain complex of left $\k[G]$-modules as a chain complex with a $G$-\textit{action} and to $\k[G]$-linear maps as $G$-\textit{equivariant}.

Given a chain complex $C$ with a $G$-action we naturally associate the following two chain complexes.
The subcomplex of \textit{invariant chains} of $C$, denoted $C^G$, contains all elements $c \in C$ satisfying $g \cdot c = c$ for every $g \in G$.
The quotient complex of \textit{coinvariant chains} of $C$, denoted $C_G$, is the chain complex obtained by identifying elements $c, c^\prime \in C$ if there exists $g \in G$ such that $c^\prime = g \cdot c$.

Let $C$ and $C^\prime$ be chain complexes and assume $C$ has a $G$-action.
The chain complex $\Hom(C, C^\prime)$ has a $G$-action induced from $(g \cdot f)(c) = f(g^{-1} \cdot c)$ and there is an isomorphism
\begin{equation} \label{e:invariant hom iso hom coinvariants}
\Hom(C, C^\prime)^G \cong \Hom(C_G, C^\prime).
\end{equation}

\subsection{Simplicial topology}

Simplicial complexes are used to combinatorially encode the topology of spaces.
An \textit{abstract and ordered simplicial complex}, or a \textit{simplicial complex} for short, is a pair $(V, X)$ with $V$ a poset and $X$ a set of subsets of $V$ such that:
\begin{enumerate}
	\item The restriction of the partial order of $V$ to any element in $X$ defines a total order on it.
	\item For every $v$ in $V$, the singleton $\{v\}$ is in $X$.
	\item If $x$ is in $X$ and $y$ is a subset of $x$, then $y$ is in $X$.
\end{enumerate}
We abuse notation and denote the pair $(V, X)$ simply by $X$ referring to $V$ as its poset of vertices.

The elements of $X$ are called \textit{simplices} and the \textit{dimension} of a simplex $x$ is defined by subtracting $1$ from the number of vertices it contains.
Simplices of dimension $n$ are called $n$-simplices and are denoted by their order set of vertices $[v_0, \dots, v_n]$.
The collection of $n$-simplices of $X$ is denoted $X_n$.
There are natural maps $d_i^n \colon X_n \to X_{n-1}$ for $i \in \{0, \dots, n\}$ defined by
\begin{equation*}
d_i^n \big( [v_0, \dots, v_n] \big) =
[v_0, \dots, \widehat{v}_i, \dots, v_n]
\end{equation*}
and referred to as the $i^\th$ face map in dimension $n$.
These satisfy the \textit{simplicial relation}:
\begin{equation} \label{e:simplicial relation}
d_i^{n-1} d^n_j = d_{j-1}^{n-1} d_i^n
\end{equation}
for any $0 \leq i < j \leq n$.
We will omit the superscripts of these maps when no confusion arises from doing so.

A \textit{simplicial map} $X \to X^\prime$ is a morphisms between their posets of vertices $f \colon V \to V^\prime$ sending simplices to simplices, i.e., satisfying that if $[v_0, \dots, v_n] \in X$ then the set $\{f(v_0), \dots, f(v_n)\}$ defines a simplex in $X^\prime$.

Let $X$ be simplicial complex.
The degree-$n$ part of the chain complex of \textit{chains} of $X$ is defined by
\begin{equation*}
C_n(X; \k) = \k \big\{ X_n \big\},
\end{equation*}
i.e., the $\k$-module freely generated by the $n$-dimensional simplices of $X$.
The \mbox{degree-$n$} part of the boundary map $\bd$ is the linear map defined on simplices by
\begin{equation*}
\begin{tikzcd}[column sep=normal, row sep=tiny,row sep = 0ex
,/tikz/column 1/.append style={anchor=base east}
,/tikz/column 2/.append style={anchor=base west}
]
C_n(X; \k) \arrow[r, "\partial_n"] & C_{n-1}(X; \k) \\
x\ \arrow[r, |->] & \sum_{i=0}^{n} (-1)^i d_i (x).
\end{tikzcd}
\end{equation*}

Given a simplicial map $f \colon X \to X^\prime$, the \textit{induced chain map} $f_\bullet \colon C_\bullet(X; \k) \to C_\bullet(X^\prime; \k)$ is defined on simplices by $f_\bullet([v_0, \dots, v_n]) = [f(v_0), \dots, f(v_n)]$ if $i \neq j$ implies $ f(v_i) \neq f(v_j)$ and it is $0$ otherwise.

We refer to
\begin{equation*}
C^\bullet(X; \k) = \Hom \big( C_\bullet(X; \k), \k \big)
\end{equation*}
as the \textit{cochains} of $X$ and to the dual $\delta^{n}$ of $\partial_{-n}$ as the $n^\th$ coboundary map.
Furthermore, we denote the linear dual of the map $f_\bullet$ induced by a simplicial map $f$ by $f^\bullet$.
We remark that $C_n(X; \k) = 0$ for $n < 0$ and $C^n(X;\k) = 0$ for $n > 0$.
Elements in the kernel of $\delta_n$ are called \textit{cocycles} and those in the image of $\delta_{n+1}$ \textit{coboundaries}.
The $n^\th$-\textit{cohomology} $H^n(X; \k)$ of $X$ is the quotient $\ker \delta_n / \img \delta_{n+1}$.
We denote by $[\alpha]$ the cohomology class represented by a cocycle $\alpha$.

We will abuse notation and identify simplices in $X_n$ with their associated basis elements in $\chain{n}(X; \k)$.
When $X$ and $\k$ are clear from the context we will omit them from the notation.

\section{Cup-\texorpdfstring{$i$}{i} constructions and Steenrod squares} \label{s:squares}

Let $\Ftwo$ be the field with two elements and $\Sym_2$ the group with only one non-identity element $T$.
In this section we define for any simplicial complex $X$ and every integer $k$ the $k^\th$ Steenrod square
\begin{equation*}
Sq^k \colon H^\bullet(X; \Ftwo) \to H^{\bullet}(X; \Ftwo)
\end{equation*}
using an arbitrary cup-$i$ construction.

\subsection{Cup-$i$ constructions}

Consider the chain complex
\begin{equation*}
\begin{tikzcd}[column sep=normal]
& W =  \Ftwo[\Sym_2]\{e_0\} & \arrow[l, "1+T"'] \Ftwo[\Sym_2]\{e_1\} & \arrow[l, "1+T"']
\Ftwo[\Sym_2]\{e_2\} & \arrow[l, "1+T"'] \cdots
\end{tikzcd}
\end{equation*}
with its natural $\Sym_2$-action.
For any simplicial complex $X$, the chain complex $W \ot C_\bullet(X; \Ftwo)$ has an $\Sym_2$-action concentrated on the left factor, and $C_\bullet(X; \Ftwo)^{\ot 2}$ has one given by transposition of factors.

We are interested in $\Sym_2$-equivariant chain maps
\begin{equation} \label{e:steenrod diagonal}
\triangle_X \colon W \ot C_\bullet(X; \Ftwo) \to C_\bullet(X; \Ftwo)^{\ot 2}
\end{equation}
defined naturally for every simplicial complex $X$, i.e., such that $\triangle_Y \circ (\id_W \ot f_\bullet) = (f_\bullet \ot f_\bullet) \circ \triangle_X$ for any simplicial map $f \colon X \to Y$.

\begin{definition}
	A (non-degenerate) \textit{cup-$i$ construction} is a natural collection of maps as above such that $\triangle_X \neq 0$ if $X$ is a simplicial complex with a single vertex.
\end{definition}

A cup-$i$ construction is determined by a collection $\{\triangle_i\}_{i \in \Z}$ of natural linear maps $C_\bullet \to C_\bullet^{\ot 2}$ satisfying $\triangle_0 \big([v]\big) \neq 0$ for any vertex $v$ and
\begin{equation} \label{e:boundary of cup-i}
(1 + T) \triangle_{i-1} = \partial \circ \triangle_i + \triangle_i \circ \partial
\end{equation}
for any $i \in \Z$.
The correspondence is given by $\triangle_i = \triangle(e_i \ot -)$, and we refer to the map $\triangle_i$ as the \textit{cup-$i$ coproduct} of the cup-$i$ construction, and to the linear dual $\smallsmile_i$ of $\triangle_i$ as its \textit{cup-$i$ product}.
Explicitly, given two cochains $\alpha$ and $\beta$ and a chain $c$ we have
\begin{equation*}
(\alpha \smallsmile_i \beta)(c) = (\alpha \ot \beta) \triangle_i(c).
\end{equation*}

\subsection{Steenrod squares}

Let us consider a cup-$i$ construction $W \ot C_\bullet \to C_\bullet^{\ot 2}$.
Using the linear duality functor and passing to fix points it gives a chain map
\begin{equation*}
\begin{tikzcd}
\Hom\left(C_\bullet \ot C_\bullet, \Ftwo \right)^{\Sym_2} \arrow[r] &
\Hom\left(W \ot C_\bullet, \Ftwo \right)^{\Sym_2},
\end{tikzcd}
\end{equation*}
which we can complete, using isomorphisms \eqref{e:adjunction isomorphism} and \eqref{e:invariant hom iso hom coinvariants} of \cref{s:preliminaries}, to a commutative diagram
\begin{equation*}
\begin{tikzcd}
\Hom\left(C_\bullet \ot C_\bullet, \Ftwo \right)^{\Sym_2} \arrow[r] &
\Hom\left(W \ot C_\bullet, \Ftwo \right)^{\Sym_2} \arrow[d] \\
\left(C^\bullet \ot C^\bullet\right)^{\Sym_2} \arrow[u]&
\Hom\left(W_{\Sym_2} \ot C_\bullet, \Ftwo \right) \arrow[d] \\
C^\bullet \arrow[u, "doubleing"] \arrow[r, dashed]&
\Hom\left(W_{\Sym_2}, C^\bullet\right),
\end{tikzcd}
\end{equation*}
where the choice of coefficients ensures that the \textit{doubleing map} $\alpha \mapsto \alpha \ot \alpha$ is linear.
Using the adjunction isomorphism, the dashed arrow defines a linear map
\begin{equation} \label{e:Steenrod squares parameterized}
\begin{tikzcd}[row sep=0pt, column sep = small]
C^\bullet \ot W_{\Sym_2} \arrow[r] &[-10pt] C^\bullet \\
\alpha \ot e_i \arrow[r, |->] & (\alpha \ot \alpha)\triangle_i(-)
\end{tikzcd}
\end{equation}
descending to mod $2$ cohomology.
As described below, the Steenrod squares are defined by reindexing this map.
\begin{definition} \label{d:steenrod squares}
	The \textit{$k^\th$ Steenrod square} is defined by
	\begin{equation} \label{e:steenrod squares}
	\begin{tikzcd}[row sep=0pt, column sep=tiny]
	Sq^k \colon H^{-n} \arrow[r] & H^{-n-k} \\
	\phantom{Sq^k \colon}{[\alpha]} \arrow[r, |->] & \big[ (\alpha \ot \alpha)\triangle_{n-k}(-) \big].
	\end{tikzcd}
	\end{equation}
	for any cup-$i$ construction $\triangle$.
\end{definition}

\subsection{Additional comments}

\begin{remark}[Simplicial sets]
	For the interested reader we mention that a cup-$i$ construction also defines, through a well known categorical construction, natural cup-$i$ coproducts on the chains of simplicial sets \cite{friedman2012simplicial} and, consequently, Steenrod squares in their mod 2 cohomology.
\end{remark}

\begin{remark}[Cup product] \label{r:cup product}
	Although in this article we do not use the algebra structure on the mod~2 cohomology of spaces, we remark that the cup-$0$ product of a cup-$i$ construction represents the \textit{cup product} in cohomology.
	Explicitly, if $[\alpha], [\beta] \in H^\bullet$ then $[\alpha][\beta] = [\alpha \smallsmile_0 \beta]$, in particular, if $[\alpha]$ is of degree $-k$ then $Sq^k\big([\alpha]\big) = [\alpha] [\alpha]$, which motivates the term squares in the name of the $Sq^k$ operations.
\end{remark}

\begin{remark}[Transverse intersections]
	From a geometric viewpoint, the cup product can be interpreted in terms of intersections of cycles in certain cases.
	For any space, Thom showed that every mod $2$ homology class is represented by the push-forward of the fundamental class of a closed manifold $W$ along some map to the space.
	Furthermore, if the target $M$ is a closed manifold, and therefore satisfies Poincar\'{e} duality
	\[
	PD \colon H^k(M ;\Ftwo) \to H_{\bars{M}-k}(M; \Ftwo),
	\]
	the cohomology class dual to the homology class represented by the intersection of two transverse maps $V \to M$ and $W \to M$, or more precisely their pull-back $W \times_M V \to M$, is the cohomology class $[\alpha] [\beta]$ where $[\alpha]$ and $[\beta]$ are respectively dual to the homology classes represented by $V \to M$ and $W \to M$.
	By taking $[\alpha] = [\beta]$ we have that $Sq^k \big( [\alpha] \big)$ with $\alpha$ of degree $-k$ is represented by the transverse self-intersection of $W \to M$, that is, the intersection of this map and a generic perturbation of itself.
	In manifold topology, the relationship at the (co)homology level between cup product and intersection is classical.
	For a comparison between these at the level of (co)chain see \cite{medina2021flowing}.
	A generalization of this result to cup-$i$ products is the focus of current research.
\end{remark}

\begin{remark}[Odd primes]
	For the reader familiar with group homology, we remark that Steenrod squares are parameterized by classes on the mod $2$ homology of $\Sym_2$.
	Steenrod used this group homology viewpoint to non-constructively define operations on the mod $p$ cohomology of spaces \cite{steenrod1952reduced, steenrod1953cyclic, steenrod1962cohomology} for any prime $p$.
	To define these constructively, analogues of explicit cup-$i$ coproducts for odd primes were introduced in \cite{medina2021may_st} using May's operadic viewpoint \cite{may1970general} and implemented in the computer algebra system \texttt{ComCH} \cite{medina2021comch}.
\end{remark}

\section{New formulas for cup-\texorpdfstring{$i$}{i} products} \label{s:formulas}

In this section we introduce formulas which we show to define a cup-$i$ construction in \cref{s:proof}.
To the best of our knowledge these are new expressions.
In forthcoming work \cite{medina2022axiomatic} we prove that the resulting cup-$i$ construction agrees up to isomorphism with Steenrod's original and all other cup-$i$ constructions in the literature.

\begin{notation}
	Let $X$ be a simplicial complex and $x \in X_n$.
	For a set
	\begin{equation*}
	U = \{u_1 < \dots < u_r\} \subseteq \{0, \dots, n\}
	\end{equation*}
	we write $d_U(x) = d_{u_1}\! \dotsm \, d_{u_r}(x)$, with $d_{\emptyset}(x) = x$.
\end{notation}

\begin{definition} \label{d:cup-i coproducts}
	For any simplicial complex $X$ and integer $i$
	\begin{equation*}
	\Delta_i \colon C_\bullet(X; \Ftwo) \to C_\bullet(X; \Ftwo) \ot C_\bullet(X; \Ftwo)
	\end{equation*}
	is the linear map defined on a simplex $x \in X_n$ to be $0$ if $i \not\in \{0, \dots, n\}$ and is otherwise given by
	\begin{equation} \label{e:new formulas}
	\Delta_i(x) = \sum d_{U^0}(x) \ot d_{U^1}(x)
	\end{equation}
	where the sum is taken over all subsets $U = \{u_1 < \cdots < u_{n-i}\} \subseteq \{0, \dots, n\}$ and
	\begin{equation} \label{e:partition subsets}
	U^0 = \{u_j \in U\mid u_j \equiv j \text{ mod } 2\}, \qquad
	U^1 = \{u_j \in U\mid u_j \not\equiv j \text{ mod } 2\}.
	\end{equation}
\end{definition}

\begin{example} \label{ex:alexander-whitney diagonal}
	For any $x \in X_n$ and $i = 0$ our formulas give
	\begin{equation*}
	\Delta_0(x) = \sum_{j=0}^n d_{j+1} \cdots d_{n}(x) \ot d_{0} \cdots d_{j-1}(x),
	\end{equation*}
	a map known as \textit{Alexander--Whitney diagonal} and widely used to define the algebra structure on cohomology (\cref{r:cup product}).
\end{example}

\begin{example} \label{ex:Sq0 is the identity}
	For any simplex $x \in X_n$ our formulas give
	\begin{equation*}
	\Delta_n(x) = x \ot x,
	\end{equation*}
	implying, after \cref{t:main} below, the well known fact that $Sq^0$ is the identity.
\end{example}

\begin{theorem} \label{t:main}
	The maps introduced in \cref{d:cup-i coproducts} define a cup-$i$ construction.
\end{theorem}

\begin{remark}
	Two \mbox{cup-$i$} constructions, say $\triangle$ and $\triangle^\prime$, are \textit{isomorphic} if there is an automorphism $\phi$ of $W$ making the following diagram commute:
	\[
	\begin{tikzcd} [column sep = 0, row sep=normal]
	W \ot C_\bullet \arrow[rr, "\phi \ot \id"] \arrow[rd, in=180, out=-90, "\triangle^{\phantom{\prime}}"', near start] & &
	W \ot C_\bullet \arrow[ld, in=0, out=-90, "\triangle^\prime", near start] \\
	& C_\bullet \ot C_\bullet & .
	\end{tikzcd}
	\]
	The cup-$i$ products of Steenrod seem to be combinatorially fundamental.
	In forthcoming work \cite{medina2022axiomatic} that depends on \cref{t:main} we show, through an axiomatic characterization, that all known cup-$i$ constructions on simplicial chains are isomorphic -- and not just homotopic -- to the one introduced here.
	These constructions are: Steenrod's original \cite{steenrod1947products}, the one obtained using the $\EZ$-$\AW$ contraction \cite{real1996computability, gonzalez-diaz1999steenrod}, those from combinatorial operads \cite{mcclure2003multivariable, berger2004combinatorial}, and the one defined by the $\Med$-bialgebra structure on standard simplices \cite{medina2020prop1, medina2021prop2}.
	Furthermore, this cup-$i$ construction defines naturally another fundamental construction: the nerve of higher categories \cite{street1987orientals, medina2020globular}.
\end{remark}

In order to prove \cref{t:main} we need to check that each $\Delta_i$ is natural and satisfies \eqref{e:boundary of cup-i} -- \cref{ex:Sq0 is the identity} implies the non-degeneracy condition.
We state these claims as two lemmas.

\begin{lemma} \label{l:naturality}
	For any simplicial map $f$ and integer $i$ we have
	\begin{equation*}
	\Delta_i \circ f_\bullet = (f_\bullet \ot f_\bullet) \circ \Delta_i.
	\end{equation*}
\end{lemma}

\begin{proof}
	Consider a simplex $x = [v_0, \dots, v_n]$ and let $i \in \{0, \dots, n\}$, otherwise the identity holds trivially.
	First assume that $f_\bullet(x)$ is not $0$.
	Then, for any proper subset $U \subsetneq \{0, \dots, n\}$ the image of $d_U(x)$ is not $0$ as well and we have
	\begin{align*}
	\Delta_i \circ f_\bullet(x) =\ &
	\Delta_i \big([f(v_0), \dots, f(v_n)]\big) \\ =\ &
	\sum d_{U^0} \big([f(v_0), \dots, f(v_n)]\big) \ot d_{U^1} \big([f(v_0), \dots, f(v_n)]\big) \\ =\ &
	(f_\bullet \ot f_\bullet) \sum d_{U^0} \big([v_0, \dots, v_n]\big) \ot d_{U^1} \big([v_0, \dots, v_n]\big) \\ =\ &
	(f_\bullet \ot f_\bullet) \circ \Delta_i(x).
	\end{align*}
	If $f_\bullet(x) = 0$ then there exists consecutive elements $v_j$ and $v_{j+1}$ with $f(v_j) = f(v_{j+1})$.
	To prove that $(f_\bullet \ot f_\bullet) \circ \Delta_i(x) = 0$ it suffices to show that for any $U \in \P_{n-i}(n)$ either the simplex $d_{U^0}(x)$ or $d_{U^1}(x)$ contains both $v_j$ and $v_{j+1}$.
	If $U$ does not contain both $j$ and $j+1$ this is immediate.
	If it does, we have that $j, j+1 \in U^0$ or $j, j+1 \in U^1$ since they are consecutive implying $v_j, v_{j+1} \in d_{U^1}(x)$ in the first case and $v_j, v_{j+1} \in d_{U^0}(x)$ in the second.
\end{proof}

\begin{lemma} \label{l:main}
	For any integer $i$ we have
	\begin{equation*}
	\partial \circ \Delta_{i} + \Delta_{i+1} \circ \partial = (1+T) \Delta_{i-1}.
	\end{equation*}
\end{lemma}

We devote \cref{s:proof} to the proof of this lemma.
We now turn to the development of a fast method for the computation of Steenrod squares on the cohomology of finite simplicial complexes leveraging formula \eqref{e:new formulas}.

\section{New algorithm for Steenrod squares} \label{s:algorithm}

\begin{figure}[b]
	\input{aux/stsq}
	\caption{Let $X$ be a simplicial complex $X$.
		Passing the support $A \subseteq X_n$ of a cocycle $\alpha$ and an integer $k \in \{1, \dots, n\}$, the algorithm returns the support $B \subseteq X_{n+k}$ of a cocycle representing $Sq^k \big( [\alpha] \big)$.
		We use the notation $S \xor S^\prime = S \cup S^\prime \setminus (S \cap S^\prime)$ and $\ind(S) = \{\ind(v) \mid v \in S\}$.}
	\label{f:algorithm}
\end{figure}

For a finite simplicial complex $X$, integer $k$ and cocycle $\alpha$ of degree $-n$, the cocycle $\beta = (\alpha \ot \alpha)\Delta_{n-k}(-)$ is by
\cref{d:steenrod squares} and \cref{t:main} a representative of $Sq^k \big( [\alpha] \big)$.
In this section we will present and discuss an algorithmic description of $\supp \beta$, the support of $\beta$.

Let $A = \{a_1, \dots, a_m\} \subseteq X_n$ be the support of $\alpha$, which is defined by
\[
\alpha(x) = \begin{cases}
1 & x \in A, \\ 0 & x \not\in A,
\end{cases}
\]
for any $x \in X$.

If $k < 0$ or $k > n$, we have $\beta = 0$ by definition,
so $\supp \beta = \emptyset$.
If $k = 0$, \cref{ex:Sq0 is the identity} shows that $\beta = \alpha$, so $\supp \beta = A$.
For the remaining cases we have the following characterization whose proof occupies \cref{s:correctness}.

\begin{theorem} \label{t:algorithm}
	Let $B$ be the output of \cref{a:algorithm} when the input is $A$ and $k$, then $\supp \beta = B$.
\end{theorem}

We now give an intuitive comparison between our proposed method and a more direct approach using a generic presentation of a cup-$i$ construction
\[
\triangle_i(x) =
\sum_{\Gamma_i} x^{(1)} \ot x^{(2)}.
\]
An algorithm for the computation of the support of $(\alpha \ot \alpha) \triangle_{n-k}(-)$ can be defined by looping over $X_{n+k}$ times $\Gamma_{n-k}$ while evaluating $(\alpha \ot \alpha)$ on the associated tensor pair.
\cref{a:algorithm} improves on this scheme by using the specific form of \eqref{e:new formulas} to filter summands using the support of $\alpha$.
So, even if $X_{n+k}$ and $\Gamma_{n-k}$ are very large, \cref{a:algorithm} loops over
\[
\frac{m(m-1)}{2}
\]
unordered pairs of distinct simplices, where $m$ is the cardinality of $\supp \alpha$.
Many of these pairs are discarded quickly, after checking that the union of its simplices does not have exactly $n+k$ vertices.
One could wonder if the next step in \cref{a:algorithm} -- determining if a resulting set of $n+k$ vertices is a simplex of $X$ -- could slow down the routine significantly.
As illustrated in \cref{s:comparison} through an example, even for a sub-optimal implementation of our algorithm this is not the case.
For high-performance tasks this look-up time could be further reduced by using data structures specialized on the representation of simplicial complexes, but we do not discuss these optimizations here.

\section{Correctness of \texorpdfstring{\cref{a:algorithm}}{Algorithm 1}} \label{s:correctness}

Let us consider the same setup as above. Explicitly, a simplicial complex $X$, a cocycle $\alpha$ whose support is $A = \{a_1, \dots, a_m\} \subseteq X_n$ and an integer $k \in \{1, \dots, n\}$.
Denote by $\alpha_i$ the cochain dual of $a_i$, and consider
$\Delta_{n-k}$ as in \cref{d:cup-i coproducts}.

Before proving \cref{t:algorithm}, the correctness of \cref{a:algorithm}, let us record a few properties satisfied by our cup-$i$ construction.

\begin{lemma} \label{l:freeness}
	For $i \neq j$ and $x \in X_{n+k}$:
	\begin{enumerate}
		\item $(\alpha_i \ot \alpha_i)\Delta_{n-k}(x) = 0$.
		\item If $(\alpha_i \ot \alpha_j)\Delta_{n-k}(x) \neq 0$ then $(\alpha_j \ot \alpha_i)\Delta_{n-k}(x) = 0$.
		\item If $(1+T)(\alpha_i \ot \alpha_j)\Delta_{n-k}(x) \neq 0$ then $x = a_i \cup a_j$.
	\end{enumerate}
\end{lemma}

\begin{proof}
	Recall that
	\begin{equation*}
	\Delta_{n-k}(x) \ = \! \sum_{\substack{U \subseteq \{0, \dots, n+k\} \\ \vert U \vert = 2k}}
	d_{U^0}(x) \ot d_{U^1}(x).
	\end{equation*}

	\begin{enumerate}
		\item If $(\alpha_i \ot \alpha_i)\Delta_{n-k}(x) \neq 0$, then there exists a non-empty $U$ in the sum with $U^0 = U^1$, which is impossible since $U^0 \cap U^1 = \emptyset$.

		\item If $(\alpha_i \ot \alpha_j)\Delta_{n-k}(x) \neq 0$ and $(\alpha_j \ot \alpha_i)\Delta_{n-k}(x) \neq 0$, then there are distinct subsets $V$ and $W$ in the sum such that $V^0 = W^1$ and $W^0 = V^1$.
		But then $V = V^0 \cup V^1 = W^1 \cup W^0 = W$, which is a contradiction.

		\item If $(1+T)(\alpha_i \ot \alpha_j)\Delta_{n-k}(x) \neq 0$, then there exists $U \subseteq \{0, \dots, n+k\}$ of cardinality $2k$ such that $\{a_i, a_j\} = \{d_{U^0}(x), d_{U^1}(x)\}$ and, since $U^0 \cap U^1 = \emptyset$, we have $x = d_{U^0}(x) \cup d_{U^1}(x)$.
		The claim follows.
	\end{enumerate}
\end{proof}

We will need the following functions.

\begin{definition} \label{d:position function}
	Given a finite totally ordered set $S$, the \textit{position function} $\pos_S \colon S \to \N$ sends an element $s \in S$ to the cardinality of $\{s^\prime \in S \mid s^\prime \leq s\}$.
\end{definition}

\begin{definition} \label{d:index function}
	For $U = \{u_1 < \cdots < u_m\} \subseteq \N$ the \textit{index function} is defined by
	\[
	\begin{split}
	\ind_U \colon U & \to \Ftwo \\
	u_j & \mapsto (u_j + j) \modulo 2.
	\end{split}
	\]
\end{definition}

We can use the index function to give the following characterization of \eqref{e:partition subsets} in the definition of our cup-$i$ construction.

\begin{lemma} \label{l:partition via index function}
	For any finite set $U \subset \N$
	\[
	U^0 = \ind_U^{-1}(0), \qquad
	U^1 = \ind_U^{-1}(1).
	\]
\end{lemma}

\begin{notation}
	We will use the following notational conventions:
	\begin{enumerate}
		\item For any function $f$ and $S \subseteq \domain(f)$
		\[
		f(S) = \{ f(s) \mid s \in S\}.
		\]
		\item For any two sets $S$ and $S^\prime$
		\[
		S \xor S^\prime = S \cup S^\prime \setminus (S \cap S^\prime).
		\]
	\end{enumerate}
\end{notation}

\begin{proof}[Proof of \cref{t:algorithm}]
	We have to show that $\supp \beta = B$, where $\beta = (\alpha \ot \alpha) \Delta_{n-k}$ and $B$ is the output of \cref{a:algorithm} when the input is $A$ and $k$.

	Using (1) in \cref{l:freeness}, for any $x \in X_{n+k}$ we have that
	\begin{equation} \label{e:proof correctness}
	\begin{split}
	\beta(x) & =
	(\alpha \ot \alpha) \Delta_{n-k}(x) \\ & =
	(\alpha_1 + \cdots + \alpha_m)^{\ot 2} \Delta_{n-k}(x) \\ & =
	\Big(\sum_{i \neq j} \alpha_i \ot \alpha_j + \sum_{i} \alpha_i \ot \alpha_i \Big)
	\Delta_{n-k}(x) \\ & =
	\Big(\sum_{i \neq j} \alpha_i \ot \alpha_j \Big)
	\Delta_{n-k}(x) \\ & =
	\sum_{i < j} (1+T) (\alpha_i \ot \alpha_j)
	\Delta_{n-k}(x)
	\end{split}
	\end{equation}
	where $\alpha_i$ is the cochain dual to $a_i$.
	By (2) and (3) in \cref{l:freeness}, for any pair $\{\alpha_i, \alpha_j\}$ the evaluation of $\alpha_i \ot \alpha_j$ or $\alpha_j \ot \alpha_i$ on $\Delta_{n-k}(x)$ is non-zero if and only if
	\[
	(1+T)(\alpha_i \ot \alpha_j) \Delta_{n-k}(x) \neq 0
	\]
	and $x$ is equal to $a_{ij} = a_i \cup a_j$.
	We say that the pair $\{\alpha_i, \alpha_j\}$ is  \textit{non-zero} in this case.
	Using these observations and \eqref{e:proof correctness},
	the support of $\beta$ can be constructed iterating over pairs $i < j$ as follows: Consider a set $B^\prime$ initialized as the empty set and update it to $B^\prime \triangle \, \{a_{ij}\} = B^\prime \cup \{a_{ij}\} \setminus (B^\prime \cap \{a_{ij}\})$ when $\{\alpha_i, \alpha_j\}$ is non-zero.
	Here we are taking advantage of the fact that cardinality mod 2 can be kept track of using the symmetric difference.
	At the end of the iteration we have $\supp \beta = B^\prime$.

	The construction of $B^\prime$ is structurally the same as that of $B$ with the exception that the condition on a pair $\{\alpha_i, \alpha_j\}$ to be non-zero is replaced by an \textbf{if} condition in terms of the pair $\{a_i, a_j\}$ only.
	The theorem will follow after showing that these two conditions are equivalent.

	A pair $\{\alpha_i, \alpha_j\}$ is non-zero if and only if there exists $U \subseteq \{0, \dots, n+k\}$ of cardinality $2k$ such that
	\[
	\{a_i, a_j\} =\{d_{U^0}(a_{ij}), d_{U^1}(a_{ij})\}.
	\]
	If such $U$ exists it is unique, and it is the image under the position function $\pos_{a_{ij}} \colon {a}_{ij} \to \N$ of the subset $\overline{a}_{ij}$ defined by
	\begin{equation*}
	\overline{a}_{i} = a_i \setminus a_j, \qquad
	\overline{a}_{j} = a_j \setminus a_i, \qquad
	\overline{a}_{ij} = \overline{a}_i \cup \overline{a}_j.
	\end{equation*}
	Therefore, a pair $\{\alpha_i, \alpha_j\}$ is non-zero if an only if for $U = \pos_{a_{ij}}(\overline{a}_{ij})$ one has
	\begin{equation} \label{e:pos's equal U's}
	\big\{\pos_{a_{ij}}(\overline{a}_i),\, \pos_{a_{ij}}(\overline{a}_j)\big\} = \{U^0, U^1\}.
	\end{equation}
	We now give an equivalent condition for this.
	Consider the function $\ind \colon \overline{a}_{ij} \to \Ftwo$ defined by
	\[
	\ind(v) =
	\pos_{a_{ij}}(v) + \pos_{\overline{a}_{ij}}(v) \text{ mod }2
	\]
	and notice that the following diagram
	\[
	\begin{tikzcd}[column sep=small, row sep=small]
	\overline{a}_{ij} \arrow[rr, "\pos_{a_{ij}}", "\cong"'] \arrow[dr, "\ind"', bend right] & &
	U \arrow[dl, "\ind_U", bend left] \\
	& \Ftwo &
	\end{tikzcd}
	\]
	commutes.
	Therefore, by \cref{l:partition via index function} the identity \eqref{e:pos's equal U's} holds if an only if the function $\ind \colon \overline{a}_{ij} \to \Ftwo$ is constant on both $\overline{a}_i$ and $\overline{a}_j$ with different values.
	This is equivalent to the identity
	\begin{equation*}
	\ind(\overline{a}_i) \xor \ind(\overline{a}_j) = \{0,1\},
	\end{equation*}
	as in the second \textbf{if} condition of \cref{a:algorithm}.
\end{proof}

\section{Proof-of-concept comparison} \label{s:comparison}

In this section we present a proof-of-concept comparison between the existing method for the computation of Steenrod squares on simplicial complexes, based on Gonz\'alez-D\'iaz--Real's approach \cite[Corollary 3.2]{gonzalez-diaz1999steenrod}, and the one introduced here.
We used a \verb|Python| implementation of our algorithm and the open source computer algebra system \verb|SAGE v9.3.rc3| \cite{sagemath} which includes an implementation of the existing method written by John Palmieri.
A complexity theoretic analysis goes beyond the scope of this paper.

\subsection{Suspensions of the real projective plane}

Given a topological space $X$ the \textit{suspension} of $X$ is the topological space $\suspension X$ obtained from $X \times [0,1]$ by collapsing $X \times \{0\}$ and $X \times \{1\}$ to points.
Suspension is a natural construction and, for each integer $i \neq 0$, there is an isomorphism $H^i(X) \cong H^{i+1}(\suspension X)$, which can be extended to $i = 0$ by considering reduced cohomology.
A crucial fact about Steenrod squares is that for reduced cohomology with mod 2 coefficients, all operations commuting with the suspension isomorphism are generated by the Steenrod squares.

The real projective plane $\RP^2$, obtained by identifying antipodal points in a sphere, is the simplest space with a non-trivial Steenrod square.
Its reduced mod 2 cohomology has a single basis element $x_j \in \widetilde{H}^j(\RP^2; \Ftwo)$ for $j \in \{1, 2\}$ and satisfies $Sq^1(x_1) = x_2$.
Therefore, its $i^\th$ suspension $\suspension^i \RP^2$ has a non-trivial operation given by $Sq^1(\suspension^i x_1) = \suspension^i x_2$.

\subsection{Pipeline}

We now describe the pipeline we followed for the comparison.
In \verb|SAGE| we produced a simplicial complex model of $\suspension^i \RP^2$ for each $i \in \{0, \dots, 10\}$ using the methods \verb|RealProjectiveSpace(2)| and \verb|suspension(i)|.
We used the method \verb|cohomology_ring(GF(2))| on this model and on its output the method \verb|basis()| to obtain a model for the element $\suspension^i x_1$.
Finally, we applied the method \verb|Sq(k)| to it with $k=1$ and record the execution time of this last step.
We implemented in \verb|Python| an alternative for the method \verb|Sq(k)| based on \cref{a:algorithm} and modified the above pipeline accordingly.
We recorded the average execution time of these implementations for each $\suspension^i \RP^2$ over $\lfloor 10000/2^i \rfloor$ runs for each $i \in \{0, \dots, 10\}$.
The results of this pipeline are presented in \cref{f:comparison}.

\begin{figure}
	\includegraphics[width=0.9\textwidth]{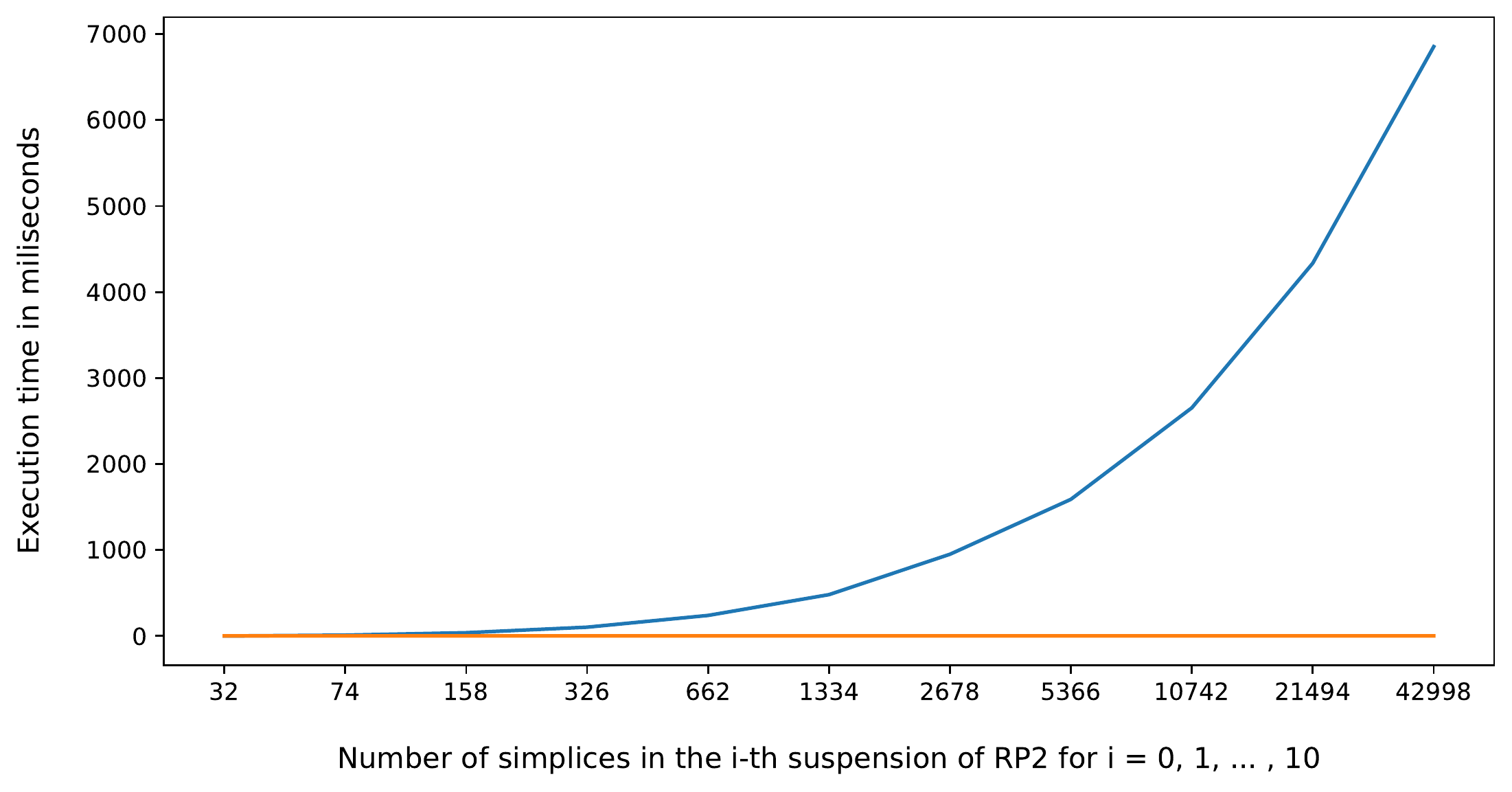}
	\caption{Average execution time in \texttt{SAGE} of two methods computing Steenrod squares. In orange the one proposed in this article and in blue the one included in \texttt{SAGE v9.3.rc3}. More specifically, for each $i \in \{0, \dots, 10\}$ we timed the computation of the non-trivial Steenrod square in the cohomology of the $i^\th$ suspension of the real projective plane, averaged over a number of runs equal to the integral part of $\frac{10000}{2^i}$.}
	\label{f:comparison}
\end{figure}

\section{Proof of \texorpdfstring{\cref{l:main}}{main lemma}} \label{s:proof}

Throughout this section $X$ denotes a simplicial complex and $\Delta_i$ the $i^\th$ map introduced in \cref{d:cup-i coproducts}.

To aid readability of the relatively long proof of \cref{l:main} we split it into four lemmas.
We start by introducing some notation.

\begin{definition}
	For $n \geq0$ and $q \in \{0, \dots, n\}$, let $\P_q(n)$ be the set of all sets $U = \{u_1 < \cdots < u_q\}$ with each $u_j \in \{0, \dots, n\}$.
	For any $U \in \P_q(n)$, let $\overline{U} \in \P_{n+1-q}(n)$ contain the elements of $\{0, \dots, n\}$ not in $U$. For $\bu \in \overline{U}$, define $\bu.U \in \P_{q+1}(n)$ to contain $\bu$ and the elements in $U$.
	For $q > 0$ and $u \in U$, define $U \sm u \in \P_{q-1}(n)$ to contain the elements of $U$ not equal to $u$.
\end{definition}

Recall that for any $U = \{u_1 < \cdots < u_q\} \in \P_q(n)$ we write $d_U$ for $d_{u_1} \cdots \; d_{u_q}$ with $d_{\emptyset} = \id$, that
the \textit{index function} of $U$ is given by
\begin{equation*}
\begin{tikzcd}[row sep=-3pt, column sep=small,
/tikz/column 1/.append style={anchor=base east},
/tikz/column 2/.append style={anchor=base west}]
\ind_U \colon U \arrow[r] & \Ftwo \\
u_i \arrow[r, mapsto] & u_i + i,
\end{tikzcd}
\end{equation*}
and that we denote the preimage of $\varepsilon \in \Ftwo \cong \{0, 1\}$ by $U^\varepsilon \subseteq U$.

With this notation, for any simplex $x \in X_n$ and $i \in \{0, \dots, n\}$ we have
\begin{equation*}
\Delta_{i}(x)\ =\! \sum_{U \in \P_{n-i}(n)} d_{U^0}(x) \ot d_{U^1}(x).
\end{equation*}

\begin{lemma} \label{l:partial dU = dxU}
	For any $x \in X_n$ and $U \in \P_{q}(n)$
	\begin{equation} \label{lemma1: existence:eq1}
	\partial_{n-q} \circ d_U(x) = \sum_{\bar{u} \in \overline{U}} d_{\bar{u}.U}(x).
	\end{equation}
\end{lemma}

\begin{proof}
	Let $U = \{u_1 < \cdots < u_q\}$. Using the simplicial relation \eqref{e:simplicial relation} we have
	\begin{equation*}
	\partial_{n-q} \circ d_U(x) =
	\sum_{i=0}^{n-q} d_i\, d_{u_1} \cdots\, d_{u_q}(x) =
	\sum_{\bar{u} \in \overline{U}} d_{u_1} \cdots\, d_{\bar{u}} \cdots\, d_{u_q}(x) =
	\sum_{\bar{u} \in \overline{U}} d_{\bar{u}.U}(x)
	\end{equation*}
	as claimed.
\end{proof}

\begin{lemma} \label{l:pigeon hole}
	For any $x \in X_n$ and $q \in \{1, \dots, n\}$
	\begin{equation} \label{e:pigeon hole 1}
	\Delta_{n-q} \circ \partial_n (x)\ =\!
	\sum_{U \in \P_q(n)} \left( \,
	\sum_{u \in U^1} d_{u.U^0} \ot d_{U^1} +
	\sum_{u \in U^0} d_{U^0} \ot d_{u.U^1} \right)(x \ot x).
	\end{equation}
\end{lemma}

\begin{proof}
	Let
	\begin{align*}
	& S_1 = \big\{ (u, V)\mid V \in \P_{q-1}(n-1) \text{ and } u \in \{0,\dots,n\} \big\}, \\
	& S_2 = \big\{ (w, W)\mid W \in \P_{q}(n) \text{ and } w \in W \big\}.
	\end{align*}
	Identity \eqref{e:pigeon hole 1} is equivalent to
	\begin{equation} \label{e:pigeon hole 2}
	\sum_{(u, V) \in S_1} d_{V^0}d_u \ot d_{V^1}d_u \ \, = \!
	\sum_{(w, W) \in S_2}
	\begin{cases}
	d_{w.W^0} \ot d_{W^1} & \text{ if } w \in W^1, \\
	d_{W^0} \ot d_{w.W^1} & \text{ if } w \in W^0.
	\end{cases}
	\end{equation}
	Define $S_1 \to S_2$ by sending $\big(u,\, \{v_1 < \cdots < v_{q-1}\} \big)$ to $\big(u,\, \{w_1 < \cdots < w_{q}\} \big)$ with
	\begin{equation*}
	w_i =
	\begin{cases}
	v_i & \text{ if } v_i < u, \\
	u & \text{ if } v_i < u \leq v_{i+1}, \\
	v_{i-1}+1 & \text{ if } v_i < u.
	\end{cases}
	\end{equation*}
	This function is a bijection since it is injective and both sets have cardinality
	\begin{equation*}
	\frac{(n+1)!}{(n+1-q)!(q-1)!}.
	\end{equation*}
	The simplicial identity implies that if $(u, V) \mapsto (u, W)$ then
	\begin{equation*}
	d_{V^0}d_u \ot d_{V^1}d_u =
	\begin{cases}
	d_{u.W^0} \ot d_{W^1} & \text{ if } u \in W^1, \\
	d_{W^0} \ot d_{u.W^1} & \text{ if } u \in W^0,
	\end{cases}
	\end{equation*}
	which concludes the proof.
\end{proof}

\begin{lemma} \label{l:boundary of Delta}
	For any $x \in X_n$ and $i \in \{1, \dots, n\}$
	\begin{equation} \label{e:boundary of Delta}
	(\partial \circ \Delta_i + \Delta_i \circ \partial)(x) \ = \!
	\sum_{\substack{U \in \P_{n-i}(n) \\ \bu \in \overline{U}}} \Big( d_{\bu.U^0} \ot d_{U^1} \, + \, d_{U^0} \ot d_{\bu.U^1} \Big) (x \ot x).
	\end{equation}
\end{lemma}

\begin{proof}
	Let $q = n-i$, we want to prove that
	\begin{equation*}
	\Big( \partial_{2n-q} \circ \Delta_{n - q} \, +\, \Delta_{n-q} \circ \partial_n \Big) (x) \ = \!
	\sum_{\substack{U \in \P_{q}(n) \\ \bu \in \overline{U}}} \Big( d_{\bu.U^0} \ot d_{U^1} \, + \, d_{U^0} \ot d_{\bu.U^1} \Big) (x \ot x).
	\end{equation*}
	Using \cref{l:partial dU = dxU} we have
	\begin{equation*}
	\begin{split}
	\partial_{2n-q} \circ \Delta_{n - q} (x) \ = & \
	\sum_{U \in \P_q(n)} \Big( \partial \circ d_{U^0} \ot d_{U^1}\ +\
	d_{U^0} \ot \partial \circ d_{U^1} \Big) (x \ot x) \\ = & \!\!
	\sum_{\substack{U \in \P_q(n) \\ \bar v \in \overline{U^0},\ \bar w \in \overline{U^1} } }\!\! \Big( d_{\bar v.U^0} \ot d_{U^1}\ +\ d_{U^0} \ot d_{\bar w.U^1}\Big) (x \ot x).
	\end{split}
	\end{equation*}
	Since for $\varepsilon \in \Ftwo$ we have a partition of $\overline{U^\varepsilon}$ into $U^{1+\varepsilon}$ and $\overline{U}$ the above can be written as
	\begin{align*}
	\partial_{2n-q} \circ \Delta_{n - q} (x) \ = &
	\sum_{U \in \P_q(n)} \left( \,
	\sum_{u \in U^1} d_{u.U^0} \ot d_{U^1} +
	\sum_{u \in U^0} d_{U^0} \ot d_{u.U^1} \right)(x \ot x) \\ + &
	\sum_{\substack{U\in\P_{q}(n) \\ \bu \in \overline{U}}} \Big( d_{\bu.U^0} \ot d_{U^1}\ +\ d_{U^0} \ot d_{\bu.U^1}\Big) (x \ot x)
	\end{align*}
	and \cref{l:pigeon hole} implies
	\begin{align*}
	\partial_{2n-q} \circ \Delta_{n - q} (x) \ =& \ \
	\Delta_{n-q} \circ \partial_n (x) \ \  \\ +&
	\sum_{\substack{U\in\P_{q}(n) \\ \bu \in \overline{U}}} \Big( d_{\bu.U^0} \ot d_{U^1}\ +\ d_{U^0} \ot d_{\bu.U^1}\Big) (x \ot x)
	\end{align*}
	which proves the claim.
\end{proof}

\begin{lemma} \label{l:big lemma}
	For any $x \in X_n$ and $i \in \{1, \dots, n\}$
	\begin{equation}
	\sum_{\substack{U \in \P_{n-i}(n) \\ \bu \in \overline{U}}} d_{\bu.U^0} \ot d_{U^1}\, +\ d_{U^0} \ot d_{\bu.U^1}\, = \
	(1+T) \Delta_{i-1}(x).
	\end{equation}
\end{lemma}

\begin{proof}
	Let $q = n-i \in \{0, \dots, n-1\}$.
	We need to show that
	\begin{equation} \label{e:big lemma}
	\sum_{\substack{U \in \P_{q}(n) \\ \bu \in \overline{U}}} d_{\bu.U^0} \ot d_{U^1}\, +\ d_{U^0} \ot d_{\bu.U^1}\, = \
	(1+T) \sum_{U \in \P_{q+1}(n)} d_{U^0} \ot d_{U^1.}
	\end{equation}
	Notice that for any $U = \{u_1 < \cdots < u_{q+1}\} \in \P_{q+1}(n)$ we have
	\begin{align*}
	& \forall u \in (U \sm u_1), \qquad \ind_{U}(u) \neq \ind_{U \sm u_1}(u), \\
	& \forall u \in (U \sm u_{q+1}), \quad \ind_{U}(u) = \ind_{U \sm u_{q+1}}(u).
	\end{align*}
	Therefore, the right hand side of \eqref{e:big lemma}
	\begin{equation} \label{e:big lemma rhs}
	\sum_{U \in \P_{q+1}(n)} d_{U^0} \ot d_{U^1}\, +\ d_{U^1} \ot d_{U^0}
	\end{equation}
	is equal to
	\begin{equation} \label{e:big lemma rhs w/o endpoints}
	\begin{split}
	&\, \sum_{\substack{U \in \P_{q+1}(n) \\ \ind_U(u_{q+1}) = 0}}
	d_{u_{q+1}.(U \sm u_{q+1})^0} \ot d_{(U \sm u_{q+1})^1}\ \ \\ +
	&\, \sum_{\substack{U \in \P_{q+1}(n) \\ \ind_U(u_{q+1}) =1}}
	d_{(U \sm u_{q+1})^0} \ot d_{u_{q+1}.(U \sm u_{q+1})^1} \\ +
	&\ \,\sum_{\substack{U \in \P_{q+1}(n) \\ \ind_U(u_1) = 1}} d_{u_1.(U \sm u_1)^0} \ot d_{(U\sm u_1)^1} \quad \\ +
	&\, \sum_{\substack{U \in \P_{q+1}(n) \\ \ind_U(u_1) = 0}} d_{(U \sm u_1)^0} \ot d_{u_1.(U\sm u_1)^1.}
	\end{split}
	\end{equation}

	With notation that will be introduced next, \eqref{e:big lemma rhs w/o endpoints} will be seen to be equal to
	\begin{equation*}
	\sum_{{L}_{max}^{e}} d_{\bu.U^0} \ot d_{U^1}\ \ +\ \
	\sum_{{R}_{max}^{o}} d_{U^0} \ot d_{\bu.U^1}\ \ +\ \
	\sum_{{L}_{min}^{o}} d_{\bu.U^0} \ot d_{U^1}\ \ +\ \
	\sum_{{R}_{min}^{e}} d_{U^0} \ot d_{\bu.U^1,}
	\end{equation*}
	and the left hand side of \eqref{e:big lemma} to
	\begin{equation*}
	\sum_{L} d_{\bu.U^0} \ot d_{U^1}\ +\ \
	\sum_{R} d_{U^0} \ot d_{\bu.U^1.}
	\end{equation*}

	For any $U = \{u_1 < \cdots < u_q\} \in \P_{q}(n)$ and $\bu \in \overline{U}$ define when possible
	\begin{equation*}
	l_U^\bu = \max \{u \in U \mid u < \bu\}, \qquad
	r_U^\bu = \min\{u \in U \mid \bu < u\},
	\end{equation*}
	and the following sets, where we use tabbing to represent inclusion and a schematic to aid readability:\\

\newcommand{\vsk}{\vspace*{3pt}}

\begin{minipage}{.6\textwidth}
	$L = \big\{ \bu.U^0 \ot U^1\ | \ U = \{u_1 < \cdots < u_{q} \} \in \P_q(n),\ \bu \in \overline{U} \big\}$ \vsk
	\begin{tab}
		$L^{e} = \{\ind_{\bu.U}(\bu) = 0\}$ \par \vsk
		\begin{tab}
			$L_{max}^{e} = \{u_q < \bu\}$ \par \vsk
			$\overline{L}_{max}^{e} = L^{e} \setminus L_{max}^{e}$ \vsk
			\begin{tab}
				$\overline{L}_{max}^{e,e} = \{\ind_{\bu.U}(r_{U}^\bu) = 0 \}$ \par \vsk
				$\overline{L}_{max}^{e,o} = \{\ind_{\bu.U}(r_{U}^\bu) = 1 \}$
			\end{tab}
		\end{tab} \vsk \vsk
		$L^{o} = \{\ind_{\bu.U}(\bu) = 1\}$ \vsk
		\begin{tab}
			$L_{min}^{o} = \{\bu < u_1 \}$ \par \vsk
			$\overline{L}_{min}^{o} = L^{o} \setminus L_{min}^{o}$ \vsk
			\begin{tab}
				$\overline{L}_{min}^{o,e} =
				\{\ind_{\bu.U}(l_{U}^\bu) = 0 \}$ \par \vsk
				$\overline{L}_{min}^{o,o} =
				\{\ind_{\bu.U}(l_{U}^\bu) = 1 \}$
			\end{tab}
		\end{tab}
	\end{tab}
\end{minipage}
\begin{minipage}{.4\textwidth}
	\begin{center}
		\begin{tikzpicture}[scale = .26]
		\node (L) at (0,0) [scale = 1.5]{$L$};
		\draw [semithick] (-6,0) -- (L) -- (6,0);
		\draw [semithick] (-6,-12) -- (6,-12) -- (6,12) -- (-6,12) -- (-6,-12);
		Statement
		\node (Le) at (0,6) [scale = 1.3] {$L^e$};
		\node (Lemax) at (-3,6) {$ {L}^e_{max}$};
		\node (-Lemax) at (3,6) {$\; \overline{L}^e_{max}$\!\!};
		\node at (3,3) [scale = .7] {$\overline{L}^{e,o}_{max}$};
		\node at (3,9) [scale = .7] {$\overline{L}^{e,e}_{max}$};

		\draw [semithick] (L) -- (Le) -- (0,12);
		\draw [semithick] (6,6) -- (-Lemax) -- (Le);

		\node (Lo) at (0,-6) [scale = 1.3] {$L^o$};
		\node (Lomin) at (-3,-6) {$ {L}^o_{min}$};
		\node (-Lomin) at (3,-6) {$\, \overline{L}^o_{min}$\!\!};
		\node at (3,-3) [scale = .7] {$\overline{L}^{o,e}_{min}$};
		\node at (3,-9) [scale = .7] {$\overline{L}^{o,o}_{min}$};

		\draw [semithick] (L) -- (Lo) -- (0,-12);
		\draw [semithick] (6,-6) -- (-Lomin) -- (Lo);
		\end{tikzpicture}
	\end{center}
\end{minipage}
\ \\

\begin{minipage}{.6\textwidth}
	$R = \big\{ U^0 \ot \bu.U^1 \mid U = \{u_1 < \cdots < u_{q}\} \in \P_q(n), \ \bu \in \overline{U} \big\}$ \vsk
	\begin{tab}
		$R^{e} = \{\ind_{\bu.U}(\bu) = 0\}$ \par \vsk
		\begin{tab}
			$R_{min}^{e} = \{u_q < \bu\}$ \par \vsk
			$\overline{R}_{min}^{e} = R^{e} \setminus R_{min}^{e}$ \vsk
			\begin{tab}
				$\overline{R}_{min}^{e,e} = \{\ind_{\bu.U}(r_{U}^\bu) = 0\}$ \par \vsk
				$\overline{R}_{min}^{e,o} = \{\ind_{\bu.U}(r_{U}^\bu) = 1\}$
			\end{tab}
		\end{tab} \vsk \vsk
		$R^{o} = \{\ind_{\bu.U}(\bu) = 1\}$ \par \vsk
		\begin{tab}
			$R_{max}^{o} = \{\bu < u_1\}$ \par \vsk
			$\overline{R}_{max}^{o} = R^{o} \setminus R_{max}^{o}$ \vsk
			\begin{tab}
				$\overline{R}_{max}^{o,e} = \{\ind_{\bu.U}(l_{U}^\bu) = 0\}$ \par \vsk
				$\overline{R}_{max}^{o,o} = \{\ind_{\bu.U}(l_{U}^\bu) = 1\}$
			\end{tab}
		\end{tab}
	\end{tab}
\end{minipage}
\begin{minipage}{.4\textwidth}
	\begin{center}
		\begin{tikzpicture}[scale = .26]
		\node (L) at (0,0) [scale = 1.5]{$R$};
		\draw [semithick] (-6,0) -- (L) -- (6,0);
		\draw [semithick] (-6,-12) -- (6,-12) -- (6,12) -- (-6,12) -- (-6,-12);

		\node (Le) at (0,6) [scale = 1.3] {$R^e$};
		\node (Lemax) at (-3,6) {$ {R}^e_{min}$};
		\node (-Lemax) at (3,6) {$\overline{R}^e_{min}$\!\!};
		\node at (3,3) [scale = .7] {$\overline{R}^{e,o}_{min}$};
		\node at (3,9) [scale = .7] {$\overline{R}^{e,e}_{min}$};

		\draw [semithick] (L) -- (Le) -- (0,12);
		\draw [semithick] (6,6) -- (-Lemax) -- (Le);

		\node (Lo) at (0,-6) [scale = 1.3] {$R^o$};
		\node (Lomin) at (-3,-6) {$ {R}^o_{max}$};
		\node (-Lomin) at (3,-6) {\,$\overline{R}^o_{max}$\!\!};
		\node at (3,-3) [scale = .7] {$\overline{R}^{o,e}_{max}$};
		\node at (3,-9) [scale = .7] {$\overline{R}^{o,o}_{max}$};

		\draw [semithick] (L) -- (Lo) -- (0,-12);
		\draw [semithick] (6,-6) -- (-Lomin) -- (Lo);
		\end{tikzpicture}
	\end{center} \vsk
\end{minipage}

	With this notation, \eqref{e:big lemma} is equivalent to
	\begin{align*}
	\sum_{L} d_{\bu.U^0} \ot d_{U^1}\ +\ \
	\sum_{R} d_{U^0} \ot d_{\bu.U^1} \ & = \
	\sum_{{L}_{max}^{e}} d_{\bu.U^0} \ot d_{U^1}\ \ +\ \
	\sum_{{R}_{max}^{o}} d_{U^0} \ot d_{\bu.U^1}\ \\ \ & +\ \
	\sum_{{L}_{min}^{o}} d_{\bu.U^0} \ot d_{U^1}\ \ +\ \
	\sum_{{R}_{min}^{e}} d_{U^0} \ot d_{\bu.U^1,}
	\end{align*}
	or, equivalently, to their difference being $0$.
	Explicitly,
	\begin{equation*}
	\sum_{\overline{L}_{max}^{e}} d_{\bu.U^0} \ot d_{U^1}\ \ +\ \
	\sum_{\overline{R}_{max}^{o}} d_{U^0} \ot d_{\bu.U^1}\ \ +\ \
	\sum_{\overline{L}_{min}^{o}} d_{\bu.U^0} \ot d_{U^1}\ \ +\ \
	\sum_{\overline{R}_{min}^{e}} d_{U^0} \ot d_{\bu.U^1} \ = \ 0,
	\end{equation*}
	which is a direct consequence of the following identities we now prove:
	\begin{equation} \label{e:big lemma four identities}
	\overline{R}_{min}^{e,e} = \overline{L}_{max,}^{e,e} \qquad
	\overline{R}_{min}^{e,o} = \overline{R}_{max,}^{o,e} \qquad
	\overline{L}_{min}^{o,e} = \overline{L}_{max,}^{e,o} \qquad
	\overline{L}_{min}^{o,o} = \overline{R}_{max.}^{o,o}
	\end{equation}

	For a pair $U \in \P_q(n)$ and $\bu \in \overline{U}$ define when possible the sets
	\begin{equation*}
	V^\bu_U = \{v_1 < \cdots < v_q\}, \qquad W_U^\bu = \{w_1 < \cdots < w_q\},
	\end{equation*}
	by
	\begin{align*}
	v_i =
	\begin{cases}
	u_i & \text{ if } u_i \neq l_U^\bu, \\
	\bu	& \text{ if } u_i = l_U^\bu,
	\end{cases}
	\qquad \quad
	w_i =
	\begin{cases}
	u_i & \text{ if } u_i \neq r_U^\bu, \\
	\bu	& \text{ if } u_i = r_U^\bu.
	\end{cases}
	\end{align*}
	Intuitively, $V_U^\bu$ is obtained from $U$ by replacing with $\bu$ the largest element in $U$ that is less than $\bu$.
	A similar description applies to $W_U^\bu$.
	When $U$ and $\bu$ are clear from the context we simplify notation writing $V$ and $W$ instead of $V^\bu_U$ and $W^\bu_U$, and $l$ and $r$ instead of $l_U^\bu$ and $r_U^\bu$.
	Notice that
	\begin{equation*}
	l.V = \bu.U = r.W
	\end{equation*}
	and that for any $u \in \bu.U$ with $u \not \in \{l,\, \bu,\, r\}$ we have
	\begin{equation*}
	\ind_{V}(u) = \ind_{U}(u) = \ind_{W}(u).
	\end{equation*}

	Let us now show that $\overline{R}_{min}^{e,e} = \overline{L}_{max}^{e,e}$.
	Consider $U^0 \ot \bu.U^1 \in \overline{R}_{min}^{e, e}$ which by definition satisfies $\ind_{\bu.U}(\bu) = \ind_{\bu.U}(l_{U}^\bu) = 0$.
	This is equivalent to $\bu \in V^1$ and $l \in U^0$.
	Therefore,
	\begin{equation*}
	U^0 \ot \bu.U^1 = \, l.V^0 \ot V^1
	\end{equation*}
	and, since $l.V^0 \ot V^1$ is an element in $\overline{L}_{max}^{e,e}$, we have $\overline{R}_{max}^{e,e} \subseteq \overline{L}_{max}^{e,e}$\,.
	Similarly, an element $\bu.U^0 \ot U^1 \in \overline{L}_{max}^{e,e}$ is equal to $W^0 \ot r.W^1 \in \overline{R}_{min}^{e,e}$ which gives the other inclusion and proves the first identity in \eqref{e:big lemma four identities}.
	The others are proven analogously, and the lemma follows.
\end{proof}

We can now provide the proof of \cref{l:main} and of our main theorem.

\begin{proof}[Proof of \cref{l:main}]
	For any integer $i$ and $x \in X_n$ we need to prove that
	\begin{equation} \label{e:main thm proof}
	(\partial \circ \Delta_{i} + \Delta_{i} \circ \partial)(x) = (1 + T) \Delta_{i-1}(x).
	\end{equation}
	If $i < 0$ or $i > n+1$ then both sides are equal to $0$ by definition.
	If $i = 0$, the right hand side of \eqref{e:main thm proof} is $0$ by definition and the left hand side is $0$ since the Alexander--Whitney diagonal is a chain map, please consult \cref{ex:alexander-whitney diagonal} for the relationship between $\Delta_0$ and this well known map.
	If $i = n+1$, then the left hand side of \eqref{e:main thm proof} is equal to $0$ by definition and the right hand side is equal to $(1+T) (x \ot x) = 0$.
	If $i \in \{0, \dots, n-1\}$, \cref{l:boundary of Delta} expresses the left hand side of \eqref{e:main thm proof} as
	\begin{equation*}
	\sum_{\substack{U \in \P_{q}(n) \\ \bu \in \overline{U}}} \Big( d_{\bu.U^0} \ot d_{U^1} \, + \, d_{U^0} \ot d_{\bu.U^1} \Big) (x \ot x),
	\end{equation*}
	whose right hand side is, thanks to \cref{l:big lemma}, equal to $(1+T)\Delta_{i-1}(x)$.
\end{proof}

\section{Secondary operations} \label{s:outlook}

Lifting relations from the (co)homology level to the (co)chain level is often a source of further (co)homological structure.
For example, cup-$i$ products provide an effective construction of coboundaries coherently enforcing the commutativity relation of the cup product in cohomology and lead to Steenrod squares.
It is natural then to wonder about what relations are satisfied by Steenrod squares themselves.
There are two notable relations to consider.
The first one, known as the \textit{Cartan relation}, expresses the interaction between these operations and the cup product:
\begin{equation*}
Sq^k \big( [\alpha] [\beta] \big) = \sum_{i+j=k} Sq^i \big([\alpha]\big)\, Sq^j \big([\beta]\big),
\end{equation*}
whereas the second, the \textit{Adem relation} \cite{adem1952iteration}, expresses dependencies appearing through iteration:
\begin{equation*}
Sq^i Sq^j = \sum_{k=0}^{\lfloor i/2 \rfloor} \binom{j-k-1}{i-2k} Sq^{i+j-k} Sq^k,
\end{equation*}
where $\lfloor- \rfloor$ denotes the integer part function and the binomial coefficient is reduced mod $2$.
To tap into the secondary structure associated with these relations, one needs to provide effective cochain level proofs for them, that is to say, construct explicit cochains enforcing them when passing to cohomology.
Such proofs were recently given respectively in \cite{medina2020cartan} and \cite{medina2021adem}, and we expect that the additional structure they unlock will also play an important role in computational topology.

\section{Conclusions and future work} \label{s:conclusion}

In this article we introduced new formulas describing cup-$i$ products on simplicial cochains over $\Ftwo$.
As proven in work being finalized \cite{medina2022axiomatic}, these formulas give raise to a cup-$i$ construction isomorphic to those introduced by Steenrod and others, but their specific form allowed us to development a fast algorithm computing Steenrod squares on the mod 2 cohomology of finite simplicial complexes.
Our method is based on the determination of the universal support of a representative of $Sq^k\big( [\alpha] \big)$ given the support of a cocycle $\alpha$, and it is therefore less impacted by the size of the simplicial complex than traditional methods that iterate over all simplices of dimension $\bars{\alpha} + k$.

In future work we will treat the general prime $p$ case.
More specifically, we will describe new formulas defining cup-$(p,i)$ products on simplicial cochains over $\Fp$.
These new formulas will lead to fast computations of mod $p$ Steenrod operations for simplicial complexes, and, as in the work of Cantero-Mor\'an \cite{cantero-moran2020khovanov} over $\Ftwo$, to the definition of Steenrod operations on Khovanov homology over $\Fp$ for a general prime $p$.

	\sloppy
	\printbibliography
\end{document}